\documentclass[reqno,11pt]{amsart}
\RequirePackage{amsmath,amssymb,amsthm,graphicx,mathrsfs,url,slashed,subcaption}
\RequirePackage[usenames,dvipsnames]{xcolor}
\RequirePackage[colorlinks=true,linkcolor=Red,citecolor=Green]{hyperref}
\RequirePackage{amsxtra}
\usepackage{cancel}
\usepackage{tikz-cd}
\usepackage{tikz}
\usepackage{xcolor}
\usepackage{soul}
\usepackage{todonotes}
\usepackage{xcolor}
\usepackage{soul}

\title{Switching Categoricity Behavior with a Given Degree}
\author{David Gonzalez, José Jeremías Valenzuela Morales, Java Darleen Villano}
\date{\today}
\subjclass[]{}
\thanks{The third author would like to thank Reed Solomon for helpful discussions in the early stages of this project.}

\newcommand{\A}{\mathcal{A}}

\newcommand{\B}{\mathcal{B}}

\renewcommand{\epsilon}{\varepsilon}
\newcommand{\<}{\langle}
\renewcommand{\>}{\rangle}

\renewcommand{\phi}{\varphi}

\newcommand{\M}{\mathcal{M}}

\newtheorem{theorem}{Theorem}[section]

\newtheorem{lemma}[theorem]{Lemma}

\newtheorem{question}[theorem]{Question}

\theoremstyle{definition}



\numberwithin{equation}{section}

\def\XXint#1#2#3{{\setbox0=\hbox{$#1{#2#3}{\int}$ }
\vcenter{\hbox{$#2#3$ }}\kern-.6\wd0}}

\usepackage[backend=bibtex,style=alphabetic,giveninits=true,maxbibnames=99]{biblatex}

\renewbibmacro{in:}{}
\DeclareFieldFormat{pages}{#1}

\begin{document}

\begin{abstract}
Given a noncomputable, computably enumerable set $D$, we construct a computable structure $G$ such that any two computable copies of $G$ are computably isomorphic while there are two $D$-computable copies of $G$ that are not $D$-computably isomorphic.
In other words, $G$ is computably categorical, but not computably categorical relative to $D$.
Conversely, we construct a computable structure that is not computably categorical, but is computably categorical relative to $D$.
These results differ from those in the literature regarding computable categoricity and its relativiziations because the degree $D$ is given instead of constructed.
Our work negatively answers a question of Downey, Harrison-Trainor and Melnikov in the computably enumerable case.
\end{abstract}

\maketitle

\section{Introduction}
Computable structure theory aims to calibrate the complexity of mathematical objects.
Isomorphism is the fundamental notion of mathematical sameness.
Given this, it is natural that understanding the complexity of isomorphisms is one of the primary subjects in computable structure theory.
A foundational observation is that sometimes two computable structures that are classically isomorphic have a computable isomorphism between them, and sometimes they do not.
We call structures with the special simplicity property that any two computable copies are computably isomorphic \textit{computably categorical}.
Computable categoricity has been of considerable interest for several decades (see, for example, \cite{Ers77,LaRo78,Sm1981,Gon1980,Rem81,GLS2003,LMMS2005,MN18} for various articles from the last 50 years studying this notion for specific classes of structures).
This definition also relativizes.
In a world where all computation is allowed to consult the oracle $X$, the relevant notion is \textit{computable categoricity relative to} $X$, which states that any $X$-computable copies have an $X$-computable isomorphism between them.
This too has been of interest for many years, often under the guise of understanding relative computable categoricity (the condition that a structure is computable categorical relative to $X$ for every $X$), or the notion of Scott rank $1$ (the condition that a structure is computable categorical relative to every $X$ on a Turing cone).
More about relative computable categoricity \cite{AKMS89,Chi90,Ven92,DKLT13,McCoy2002,McCoy2003}, and Scott rank $1$ structures \cite{Sco65,MonSR,MonEatomic,AGNHTT,GR23,GLRS,GRT} can be read in many sources including those cited here.

One might expect that a structure $\M$'s status as simple is typically unchanged when switching between these models of computation.
In other words, the notions of computable categoricity and computable categoricity relative to $X$ for different oracles $X$ should be closely tied.
In a sense, this suspicion is borne out.
For example, Goncharov \cite{Gon80EffDim} showed that if $\M$ has a copy where the $\forall\exists$ theory is decidable, then computable categoricity implies computable categoricity relative to $X$ for every $X$.
Intuitively, any ``nice'' structure that may come up in the course of typical mathematical practice has such a copy.

Remarkably, however, some structures are not quite so nice and demonstrate a constellation of wild behaviors when considering their categoricity relative to various degrees $X$.
In \cite{DHTM21}, Downey, Harrison-Trainor and Melnikov demonstrate that there is a computably categorical structure $\A$ and computably enumerable sets $X_0<_TY_0<X_1<_TY_1<\cdots$ where $\A$ is computably categorical relative to $X_i$ and not computably categorical relative to $Y_i$ for every $i$.
This was taken even further by Villano in \cite{Vil25}.
For each computable partially ordered set $P$ and computable partition $P_1\sqcup P_2=P$, she constructed a computably categorical structure $\B$ and an embedding of $P$ into the computably enumerable degrees where $\B$ is computably categorical relative to the degrees corresponding to $P_1$ and not computably categorical relative to the degrees corresponding to $P_2$.

It is natural to ask \cite{DHTM21}: is there something special about the degrees that are constructed alongside these structures with interesting categoricity behavior?
Put another way, if you received an arbitrary set $X$ instead of being allowed to construct one, can you construct a computably categorical structure that is not computably categorical relative to $X$?
Another equivalent way still would be to ask if there are any special sets $X$ such that if you are computably categorical, you are definitely computably categorical relative to $X$.
Our first main result states that, in the setting of $X$ being a c.e. set, there is nothing special about the degrees used in the constructions of \cite{DHTM21} and \cite{Vil25}.
More precisely, we demonstrate the following.

\begin{theorem}\label{thm:switchToNo}
        Given a noncomputable c.e.\ set $D$, there exists a structure $\mathcal{A}$ that is computably categorical but not computably categorical relative to $D$.
\end{theorem}

This answers Question 3.8 from \cite{DHTM21} in the case that $\mathbf{d}$ is computably enumerable.
Namely, there is no c.e. degree $\mathbf{d}$ such that every computably categorical structure is computably categorical relative to $\mathbf{d}$.

The constructions of \cite{DHTM21} and \cite{Vil25} also demonstrate the opposite phenomenon.
In particular, by moving from $Y$ to $X>_TY$ we may move from not being computably categorical relative to $Y$ to being computably categorical relative to $X$.
We also demonstrate that this behavior does not depend on the choice of $X$ among the c.e. sets.

\begin{theorem}\label{thm:switchToYes}
    Given a noncomputable c.e. set $D$, there exists a computable structure $\mathcal{A}$ that is computably categorical relative to $D$ but is not computably categorical.
\end{theorem}

These results, taken together, can be seen as the first step into making interesting behavior like the chain of \cite{DHTM21} or partial order of \cite{Vil25} using degrees that are given instead of constructed.
There are also some similarities between our constructions and those seen in \cite{DHTM21} or \cite{Vil25}.
For example, we also construct a bouquet graph, composed of centers adorned with various loop sizes, via a detailed priority construction.
Despite the apparent connection with these known results, it should be noted that new proof ideas are needed to approach the theorems in this article.
There are two primary reasons for this.
The most direct reason is that we are only able to diagonalize in the construction when the given set $X$ permits us to do so.
The other arguments were able to grant themselves permission whenever they wanted in their construction of $X$ and thereby get away with a simpler priority arrangement.
The less obvious reason is a structural one.
The constructions in \cite{DHTM21} and \cite{Vil25} always created a locally finite structure; only finitely many loops are ever added to each center.
They both note that this means that the structures are always computably categorical relative to $0'$.
In our setting, when $X=0'$, we may want to construct a structure that is not computably categorical relative to $0'$.
This means that we are necessarily outside of the locally finite domain and must deal with particular types of infinite outcomes that force daisies with infinite petals.

We deal only with the standard notions from computable structure theory referred to in this introduction.
For a general reference on computable structure theory see \cite{AK00} or \cite{MonBook1,MonBook2}.

The article is written in four sections, including the current one containing introductory material.
The second section is dedicated to the proof of Theorem \ref{thm:switchToNo}.
The third section is dedicated to the proof of Theorem \ref{thm:switchToYes}.
The final section is brief and indicates some directions for future research.

\section{Theorem 1.1}

\begin{theorem}\label{thm: c.e. permitting on the finite injury side}
    Given a noncomputable c.e.\ set $D$, there exists a computable directed graph $\mathcal{A}$ that is computably categorical but not computably categorical relative to $D$.
\end{theorem}

\subsection{Notation and definitions}
Before we list the requirements and strategies needed to prove Theorem \ref{thm: c.e. permitting on the finite injury side}, we define notation which will be utilized throughout the construction in this section and in the third section.

In our computable directed graph $\A$, the connected component with a root node called $a_{2s}$ will be referred to as the $2s$\textit{th connected component} of $\A$. The $2s$th component will always start with cycles of lengths $2$ and $5s+1$, and the $(2s+1)$st component will always start with cycles of lengths $2$ and $5s+2$. For the $2s$th and $(2s+1)$st components of $\A$, if a map $\Phi_e^D[s]$ converges on them, we let $u_s$ denote the associated $D$-use for this computation.

Since we must build $\A$ to not be computably categorical relative to $D$, we need to build a $D$-computable copy, $\B$, of $\A$. The root nodes in each connected component in $\B$ will be referred to similarly as those in $\A$. In the case where additional new cycles are attached to the root nodes $b_{2s}$ and $b_{2s+1}$ in $\B$, we define a new large value $v_s$ to be the associated $D$-use for these new cycles in $\B$. For $s\in\omega$, if both $u_s$ and $v_s$ are defined, then we call $(u_s,v_s]$ the \textit{permission interval} associated with the $2s$th and $(2s+1)$st components of $\A$. With this notation, we now cover the requirements needed for our priority construction.

\subsection{Requirements and Tree of Strategies}
We list the requirements that we will satisfy in our construction.
\[
P_e : \Phi_e^D:\mathcal{A}\to\mathcal{B} \ \text{is not an isomorphism; and}
\]
\[
S_i : \text{If $\mathcal{A}\cong\mathcal{M}_i$, then there exists a computable isomorphism $f_i:\mathcal{A}\to\mathcal{M}_i$.}
\]

Recall that $\mathcal{B}$ is a $D$-computable copy of $\mathcal{A}$ which we will also build in stages throughout the construction. The $S_i$ requirements ensure that $\mathcal{A}$ is computably categorical, and the $P_e$ requirements will ensure that $\mathcal{A}$ is not computably categorical relative to $D$.

Let $\Lambda=\{\infty<_\Lambda s <_\Lambda\dots<_\Lambda w_3<_\Lambda w_2<_\Lambda w_1<_\Lambda w_0\}$ be the set of outcomes, and let $T=\Lambda^{<\omega}$ be our tree of strategies. The construction will be performed in $\omega$ many stages $s$.

We define the \textit{current true path} $\pi_s$, the longest strategy eligible to act at stage $s$, inductively. For every $s$, $\lambda$, the empty string is eligible to act at stage $s$. Suppose the strategy $\alpha$ is eligible to act at stage $s$.
(We call such a stage an $\alpha$-stage.)
If $|\alpha|<s$, then follow the action of $\alpha$ to choose a successor $\alpha^\frown\<o\>$ on the current true path based on the outcome of that action (each type of action will be associated with an outcome below). 
If $|\alpha|=s$, then set $\pi_s=\alpha$. For all strategies $\beta$ such that $\pi_s <_L\beta$, initialize $\beta$ (i.e., set all parameters associated with $\beta$ to be undefined). If $\beta <_L \pi_s$ and $|\beta|<s$, then $\beta$ retains the same values for its parameters.

\subsection{Informal Description of the Construction and Interactions between different $P$-strategies}\label{section: interactions between P strategies}
We begin with some high-level descriptions of the strategies and their outcomes.

The $S$ strategies can take any outcome excluding $s$.
The $w_k$ correspond to waiting outcomes.
In particular, $S_i$ will try to match up the components of $M_i$ with the currently constructed components of $\A$.
The outcome $w_k$ means that it is waiting to find the $k^{th}$ equivalent component in $M_i$.
If one of the $w_k$ is a true outcome, it means that $S_i$ was never able to match up the components, so $M_i\not\cong\A$.
The $\infty$ outcome corresponds to the construction of a computable map between $M_i$ and $\A$.
Every time that $S_i$ makes progress in matching components, it will report the $\infty$ outcome.
If the $\infty$ outcome is true, it will turn out that $M_i\cong\A$ and we have an appropriate computable map witnessing this.

The $P$ strategies are a bit more complicated.
For the formal description of the $P$-strategy, please see Section \ref{formal strategy: P}.
They will use all of the outcomes.
The basic desire of a $P_e$ strategy is to find two components in $\A$ and $\B$ that look like they are matched up by $\Phi_e$ and, when permitted by $D$, switch the isomorphism type of these components (add diagonalizing loops) in $\B$ to make $\Phi_e^D$ wrong.
$D$ may not give permission on a given interval corresponding to some components, so $P_e$ will have to try many interconnected ways to get this diagonalizing situation.
The outcome $w_k$ means that $P_e$ is waiting for $\Phi_e^D$ to match up the elements associated with its $k^{th}$ attempt to diagonalize.
If this outcome is true, it will be because $\Phi_e^D$ fails to be an isomorphism, as this matching never occurs.
The outcome $s$ means success.
This outcome occurs when a successful diagonalization has been lined up and takes place.
If this is the true outcome, our intervention has worked, and so $\Phi_e^D$ fails to be an isomorphism.
The last outcome is $\infty$, which occurs when our approximation of $\Phi_e^D$ is demonstrated to be wrong, so we have to try again.
If this is the true outcome, $\Phi_e^D$ will consistently look like we need to take new action to diagonalize against it. 
In the end, however, $\Phi_e^D$ will not have a well-defined use on particular elements of $\A$, so it will not converge and will fail to be an isomorphism.

We will now detail some specific technical issues that arise as we have multiple $P$-strategies acting throughout the construction, and the solutions to those issues. 

 We begin by describing the $\infty$ outcome in more detail.
 Since we are defining permission intervals, the following can occur. Let $\alpha$ be a $P$-strategy and suppose that by stage $s-1$, it has defined its parameters $n^0_\alpha<n^1_\alpha<\dots<n^k_\alpha$ and its permission intervals $(u_{n^i_\alpha},v_{n^i_\alpha}]_{0\leq i<k}$. Suppose $D[s]\neq D[s-1]$ because there is a new number $j<u_{n^0_\alpha}$ that entered $D$ at the beginning of stage $s$. Since $j<u_{n^0_\alpha}$, this causes the computation $\Phi_e^D[s-1]$ on the $2n^0_\alpha$th and $(2n^0_\alpha+1)$st components of $\A$ to be undefined (in fact, it causes all computations $\Phi_e^D[s-1]$ on the $2n^i_\alpha$th and $(2n^i_\alpha+1)$st components of $\A$ for $0\leq i<k$ to be undefined). Since all these computations disappear, the permission intervals defined at stage $s-1$ are no longer valid. Additionally, the enumeration did not grant permission for $\alpha$ to switch the position of the diagonalizing loops in $\B$ (in fact, because $k$ was small enough, all these loops in $\B$ disappear at stage $s$ as well).

In this case, we have our $P$-strategies carry out the following actions:
\begin{enumerate}
    \item Keep all its parameters defined up to that point in the construction.
    \item Delete all permission intervals (as they are now all invalid).
    \item Reconstruct any nodes and edges in $\B$ which disappeared when $k$ entered $D$. Set the use of these nodes and edges to be the same use that they had at the previous $\alpha$-stage before the enumeration of $j$. For $\A$- and $\B$-components which possessed cycles of lengths $5n+3$ and $5n+4$ for some parameter $n$, we must do a quick homogenization step. That is, we add a cycle of length $5n+1$ to $a_{2n+1}$ and $b_{2n+1}$ and a cycle of length $5n+2$ to $a_{2n}$ and to $b_{2n}$. Note that the cycles of lengths $5n+3$ and $5n+4$ now differentiate the connected components with root nodes $a_{2n}$ and $a_{2n+1}$ (and their respective counterparts in $\B$).
    \item When eligible to act again, proceed to carry out the strategy starting with $n^0$ and take the $w_0$ outcome. Because the cycles of lengths $5n^m+3$ and $5n^m+4$ have been added to some of the $\A$ and $\B$-components for $0\leq m<k$, $\alpha$ will have to carry out its strategy with longer cycles on those affected components.
\end{enumerate}

Essentially, the actions in $(1)$ and $(2)$ allows $\alpha$ to ``start over''. Although $\alpha$ is utilizing some $\A$-components which may have cycles longer than $5n+1$ or $5n+2$ for some $n\in\omega$, they can carry out their strategies anew via $(3)$ and $(4)$. In this way, one can view this as $\alpha$ initializing itself (rather than being initialized by a higher priority strategy moving the current true path to the left). Throughout this write-up, we may refer to $\alpha$ taking the $\infty$ outcome as $\alpha$ initializing itself (and so $\alpha$ can either be initialized by a higher priority strategy moving the current true path to the left or via this new mechanism).

The $\infty$ outcome is precisely when $\alpha$ does the above actions in $(1)$-$(3)$ in response to a small enough $D$-enumeration. 

We also make the following related observation. Let $\alpha$ be a $P_e$-strategy and $\beta$ be a $P_{e'}$-strategy such that $\alpha^\frown\<\infty\>\subseteq\beta$. Then, every time $\alpha$ takes the $\infty$ outcome, $\beta$'s permission intervals (if they are defined) will remain unless it initializes itself by taking its $\infty$ outcome itself. Suppose $\alpha$ takes the $\infty$ outcome for the first time at a stage $s_0$, and so it initializes itself. Then, $\beta$ is eligible to act for the first time at stage $s_0$ and defines a parameter $n^0_\beta$ large (and so it must be larger than the maximum $v_{n^k_\alpha}$ defined by $\alpha$ before it initialized itself) and takes the outcome $w_0$. Now suppose that $\alpha$ is eligible to act again at a stage $s_1>s_0$, and so it defines a new large parameter $n^0_\alpha[s_1]>n^0_\beta[s_0]$. 

Notice that $\beta$ does not get to act again unless $\alpha$ takes the $\infty$ outcome again after stage $s_0$. If $\alpha$ does take the $\infty$ outcome again at a stage $s_2>s_1>s_0$, then it initializes itself again, and $\beta$ becomes eligible to act. When it does, it will now be working with $D[s_2]$ as its approximation of $D$ rather than $D[s_0]$. We also know that by the above, $D[s_0]\neq D[s_2]$. If $\beta$ can define parameters $n^k_\beta$ or permission intervals on $D[s_2]$, then the next time $\alpha$ is eligible to act at stage $s_3>s_2$, it will always define parameters and permission intervals to the right of what $\beta$ has defined by stage $s_3$. In particular, if $\alpha$'s progress is reset by a small enough $D$-enumeration, this $D$-enumeration will always be large enough so it does not affect $\beta$'s progress. Otherwise, $\beta$ is initialized itself by either having to take the $\infty$ outcome itself or the current true path moved to the left of $\beta$ in the tree because of an action from a higher priority strategy.
This is of particular interest in the case where the calculation associated with $\alpha$ is cofinally disrupted (i.e., $\alpha$ takes the infinite outcome in the $\liminf$) and the permission intervals get longer and longer over time.
Because of the above observations, there is no worry that a converging $\beta$ will be adversely affected by this, as eventually the disruptions to $\alpha$ will no longer injure $\beta$ as they will be too large.

We will now give formal descriptions of each strategy and their outcomes in the construction.

\subsection{Formal strategies}

\subsubsection{Global strategy}
For this construction, we assume that at any stage in the construction, at most one element can be enumerated into the set $D$. Let $s$ be the current stage of the construction. At the beginning of stage $s$, check if 
\[
D[s]=D[s-1].
\]
If so, proceed to the highest priority strategy and follow the action of all strategies along $\pi_s$. If $D[s]\neq D[s-1]$, then let $\ell$ be such that $\ell\in D[s]$ and $\ell\not\in D[t]$ for $t\leq s-1$.

\textbf{Case A}: If there is a (least) $P$-strategy $\alpha$ and $j\in\omega$ such that $\ell\in(u_{n^j_\alpha},v_{n^j_\alpha}]$, then follow the action of all strategies along $\pi_s$ until it is $\alpha$'s turn to act at stage $s$. When $\alpha$ acts, it will now be in \textbf{Subcase 2.4P} of its strategy (see near the end of Section \ref{formal strategy: P}).

\textbf{Case B}: If there is a (least) $P$-strategy $\alpha$ such that $\ell<u_{n^0_\alpha}$, then follow the action of all strategies along $\pi_s$ until it is $\alpha$'s turn to act. When $\alpha$ acts, it will now be in \textbf{Case 3P} of its strategy.

\subsubsection{$P_e$-strategies}\label{formal strategy: P}
Let $\alpha$ be a $P_e$-strategy eligible to act at stage $s$.
At any given stage, a $P_e$-strategy will have an increasing sequence of parameters $n^0<n^1<\dots<n^k$ that it has previously defined.
We will also have numbers $u_{n^j}$ and $v_{n^j}$, for $j\in\omega$, defined by $\alpha$ which we will call the \textit{associated $D$-uses}.
These numbers are associated with some $\Phi_e^D$ computation on a pair of $\A$-components (the ``uses'' $u_n$) or with newly added loops in $\B$ (the uses of these loops appearing, $v_n$).

\textbf{Case 1P}: If $\alpha$ is first eligible to act at stage $s$, define the parameter $n^0_\alpha=n^0$ to be large and take outcome $w_0$. The following components are then created in $\mathcal{A}[s]$ and in $\mathcal{B}[s]$ (See Figure 1):
            \begin{align*}
        a_{2n^0} : 2, \textcolor{red}{5n^0+1} & \ \ \ \ a_{2n^0+1} : 2, \textcolor{blue}{5n^0+2} \\
        \textcolor{darkgray}{b_{2n^0}}: 2, \textcolor{red}{5n^0+1} & \ \ \ \ \textcolor{darkgray}{b_{2n^0+1}} : 2, \textcolor{blue}{5n^0+2} .
    \end{align*}

    \begin{figure}[h]
   \centering
   \makebox[\textwidth][c]{\includegraphics[width=.7\textwidth]{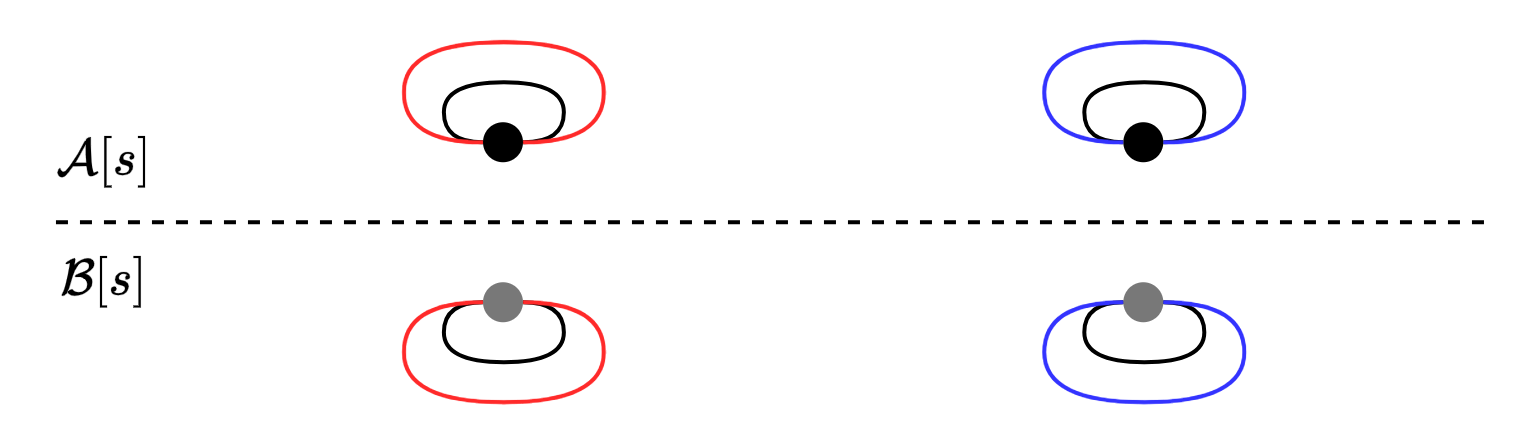}}
    \caption{}
    \end{figure}

In general, we will assume that every time a new parameter is defined the strategy constructs its associated components in $\mathcal{A}$ and $\mathcal{B}$.

\textbf{Case 2P}: When revisiting the strategy after \textbf{Case 1P} or if $\alpha$ took the $\infty$ outcome at the previous $\alpha$-stage, then suppose $\alpha$ has defined an increasing sequence of parameters $n^0<n^1<\dots<n^k$ for $k\geq 0$. We break into the following subcases.

\textbf{Subcase 2.1P}: 
$\alpha$ enters this subcase and proceeds to the description in the next paragraph if either $k=0$ or $\alpha$ defined the number $u_{n^{k-1}}$ at the previous $\alpha$-stage $t<s$ after $\Phi_e^D[t]$ mapped the $2n^{k-1}$th and $(2n^{k-1}+1)$st components of $\A$ correctly into $\B$.
Furthermore, we will require that $\alpha$ has not yet successfully defined $v_{n^{k-1}}$ as described in the next paragraph (intuitively, once we ``succeeded'' in this case for this $k$, we move on).
In this case, after defining $u_{n^{k-1}}$, $\alpha$ defines the parameter $n^k$ and takes the $w_k$ outcome at the end of stage $t$.
Check if $\Phi_e^D[s]$ maps the $2n^k$th and $(2n^k+1)$st components of $\A$ correctly into $\B$. 
If not, continue taking the $w_k$ outcome. 
If the map is correctly defined, proceed to the description in the next paragraph.

Define $u_{n^k}$ to be the maximum of the $D$-uses of the $\Phi_e^D[s]$ computation on the $2n^k$th and $(2n^k+1)$st components of $\A$.
Then, if $k>0$, add the $(5n^{k-1}+3)$-loop to $a_{2n^{k-1}}$ in $\A$ and to $b_{2n^{k-1}}$ in $\B$, and the $(5n^{k-1}+4)$-loop to $a_{2n^{k-1}+1}$ in $\A$ and to $b_{2n^{k-1}+1}$ in $\B$ and define the $D$-uses of these new loops in $\B$ to be a large unusued number $v_{n^{k-1}}$. Note that $v_{n^{k-1}}>u_{n^k}$.
After adding these new loops in $\A$, $\alpha$ now issues a challenge to all higher priority $S$-strategies $\beta$ such that $\beta^\frown\<\infty\>\subseteq\alpha$, and defines a new large parameter $n^{k+1}$.
Note that at this point, each of the challenged $S$ strategies has its parameter changed to equal $n^{k}$.
(We will discuss the parameters for $S$ strategies in detail later, but this essentially means that this strategy now only thinks it has matched components up to $n^{k}$ correctly.)
Take outcome $w_{k+1}$. The current configuration of the $2n^{k-1}$th components is (See Figure 2):

 \begin{align*}
        a_{2n^{k-1}} : 2, {5n^{k-1}+1} , \textcolor{red}{5n^{k-1}+3} & \ \ \ \ a_{2n^{k-1}+1} : 2, {5n^{k-1}+2}, \textcolor{blue}{5n^{k-1}+4} \\
        \textcolor{darkgray}{b_{2n^{k-1}}}: 2, {5n^{k-1}+1}, \textcolor{red}{5n^{k-1}+3} & \ \ \ \ \textcolor{darkgray}{b_{2n^{k-1}+1}} : 2, {5n^{k-1}+2}, \textcolor{blue}{5n^{k-1}+4}.
    \end{align*}

    \begin{figure}[h]
   \centering
   \makebox[\textwidth][c]{\includegraphics[width=.7\textwidth]{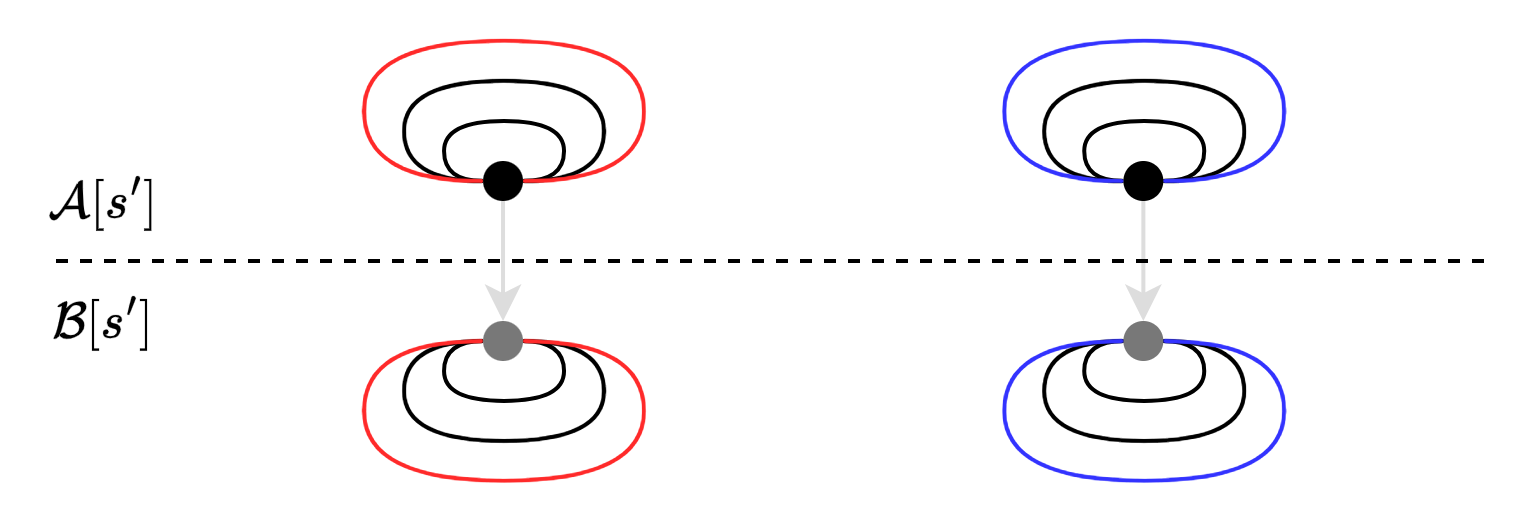}}
    \caption{}
    \end{figure}

In the case where $\alpha$ took the $\infty$ outcome at the previous $\alpha$-stage, then when we add new cycles to $\B$-components, we add cycles of lengths $5n^{k-1}+m+2$ where $5n^{k-1}+m$ was the last longest cycle length added to these $\B$-components. The $D$-uses of these new cycles, denoted by $v_{n^{k-1}}$ will be defined to be some large unused number.

\textbf{Subcase 2.2P}: Suppose $\alpha$ defined the number $v_{n^{k-2}}$ after it was able to define both $u_{n^{k-2}}$ and $u_{n^{k-1}}$ like in \textbf{Subcase 2.1P}.
In this case, $\alpha$ defined $n^k$, took outcome $w_k$ at the previous $\alpha$-stage $t<s$, and has not been initialized since stage $t$.
Since $\alpha$ can act again, then its challenge after adding new loops to $a_{2n^{k-2}}$ and to $a_{2n^{k-2}+1}$ was met by all higher priority $S$-strategies $\beta$ where $\beta^\frown\<\infty\>\subseteq\alpha$. $\alpha$ now sets up the permission interval for $n^{k-2}$, i.e., it now waits to see if $D[s']\restriction(u_{n^{k-2}},v_{n^{k-2}}]$ changes at any $\alpha$-stage $s'\geq s$. 
After setting up this permission interval, it also starts checking if $\Phi_e^D[s]$ maps the $2n^k$th and $(2n^k+1)$st components of $\A$ correctly into $\B$. 

If $\alpha$ does not see that $D[s]\restriction(u_{n^{k-2}},v_{n^{k-2}}]$ change or $\Phi_e^D[s]$ converge correctly on the $2n^k$th and $(2n^k+1)$st components of $\A$ by the end of stage $s$, continue to take the $w_k$ outcome. If $\alpha$ sees that $\Phi_e^D[s]$ converged correctly on the $\A$-components above, then define $u_{n^k}$ and a new large parameter $n^{k+1}$, and take the $w_{k+1}$ outcome.

\textbf{Subcase 2.3P}: If $\alpha$ took the waiting outcome $w_k$ at the previous $\alpha$-stage $t<s$, but no associated $D$-uses were defined during stage $t$, then continue to check if $\Phi_e^D[s]$ correctly maps the $2n^k$th and $(2n^k+1)$st components of $\A$ into $\B$. If not, continue to take outcome $w_k$. Otherwise, define $u_{n^k}$ to be the associated $D$-use for this computation, and define $v_{n^{k-1}}$ after adding the $(5n^{k-1}+3)$-loops to $a_{2n^{k-1}}$ in $\A$ and $b_{2n^{k-1}}$ in $\B$ and the $(5n^{k-1}+4)$-loops to $a_{2n^{k-1}+1}$ in $\A$ and $b_{2n^{k-1}+1}$ in $\B$. Define a new large parameter $n^{k+1}$ and issue a challenge to all higher priority $S$-strategies $\beta$ such that $\beta^\frown\<\infty\>\subseteq\alpha$, and take outcome $w_{k+1}$.

\textbf{Subcase 2.4P}: Suppose $\alpha$ took the waiting outcome $w_k$ at the previous $\alpha$-stage $t<s$. Suppose also that for some least $j<k$, we see that $D[s]\restriction(u_{n^j},v_{n^j}]\neq D[t]\restriction(u_{n^j},v_{n^j}]$. That is, permission has been granted for the parameter $n^j$. 

Since $D[s]\restriction(u_{n^j},v_{n^j}]\neq D[t]\restriction(u_{n^j},v_{n^j}]$, we have that the $(5n^j+3)$- and $(5n^j+4)$-loops disappeared in $\B$ while $\Phi_e^D[t]=\Phi_e^D[s]$ remains on the $2n^j$th and $(2n^j+1)$st components of $\A$. Add the $(5n^j+1)$- and $(5n^j+3)$-loops to $b_{2n^j+1}$ and the $(5n^j+2)$- and $(5n^j+4)$-loops to $b_{2n^j}$ in $\B$ and set the new $D$-use of these loops to be equal to the stage number $s$. Add the $(5n^j+2)$-loop to $a_{2n^j}$ and $(5n^j+1)$-loop to $a_{2n^j+1}$ in $\A$. Terminate all other progress for parameters $n'\neq n^j$, and take the success outcome $s$. The components will now have the following configuration (See Figure 3):

 \begin{align*}
        a_{2n^{k-1}} : 2, {5n^{k-1}+1} , {5n^{k-1}+2},  \textcolor{red}{5n^{k-1}+3} & \ \ \ \ a_{2n^{k-1}+1} : 2, {5n^{k-1}+1}, {5n^{k-1}+2}, \textcolor{blue}{5n^{k-1}+4} \\
        \textcolor{darkgray}{b_{2n^{k-1}}}: 2, {5n^{k-1}+1}, {5n^{k-1}+2}, \textcolor{blue}{5n^{k-1}+4} & \ \ \ \ \textcolor{darkgray}{b_{2n^{k-1}+1}} : 2, {5n^{k-1}+1}, {5n^{k-1}+2}, \textcolor{red}{5n^{k-1}+3}.
    \end{align*}

    \begin{figure}[h]
   \centering
   \makebox[\textwidth][c]{\includegraphics[width=.7\textwidth]{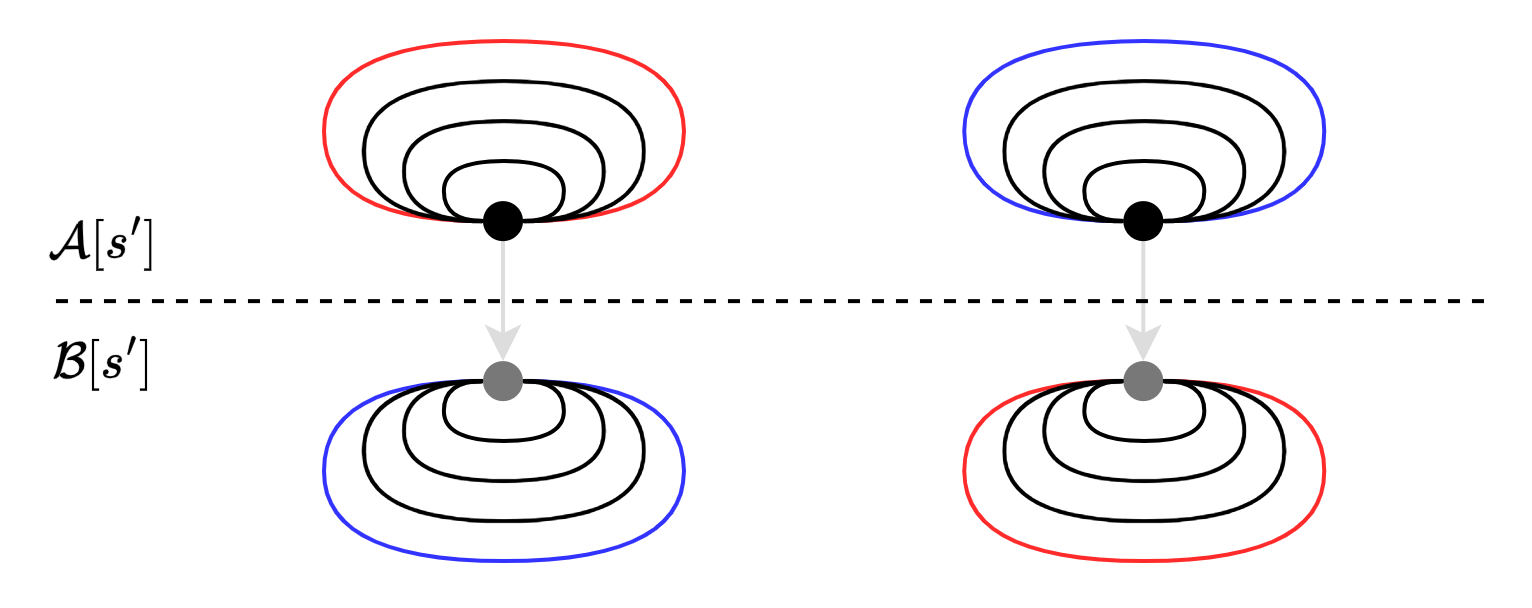}}
    \caption{}
    \end{figure}

Note that if permission is granted for $n^j$ and the associated $\A$- and $\B$-components already have cycles longer than $5n^j+1$ and $5n^j+2$, then carry out the action mentioned above with the differentiating cycles of lengths $5n^j+m+2$, where $5n^j+m$ is the longest cycle length attached only to $a_{2n^j}$ and not to $a_{2n^j+1}$ (and their respective counterparts in $\B$) and vice-versa.

\textbf{Case 3P}: Suppose $\alpha$ has defined
\[
n^0 < n^1 <\dots<n^k
\]
for $k\geq 0$ and the permission intervals 
\[
(u_{n^i},v_{n^i}]_{0\leq i<k}
\]
and took the waiting outcome $w_k$ at the previous $\alpha$-stage $t<s$. If we are in \textbf{Case B} of the global strategy, i.e., there is a small $k\in D[s]$ such that $k\not\in D[t]$ and $k<u_{n^0}$. In this case, $\alpha$ will retain the values of $n^0,\dots,n^k$ but delete the permission intervals $(u_{n^i},v_{n^i}]_{0\leq i<k}$. $\alpha$ will also reconstruct any nodes and edges in $\B$ that disappeared with the enumeration in \textbf{Case B} of the global strategy with the same use that they had before. $\alpha$ will also homogenize any $\A$- and $\B$-components as needed (see (3) in Section \ref{section: interactions between P strategies}). Take the $\infty$ outcome.

\textbf{Case 4P}: If $\alpha$ took outcome $s$ at the end of the previous $\alpha$-stage $t<s$, then continue to take outcome $s$ unless $\alpha$ is initialized.

\subsubsection{$S_i$-strategies}
Let $\alpha$ be an $S_i$-strategy eligible to act at stage $s$.

\textbf{Case 1S:} If $\alpha$ is acting for the first time or has been initialized since the last $\alpha$-stage, define the following:
\begin{itemize}
    \item $n_\alpha[s]=0$,
    \item $f_\alpha[s]$ to be the empty map, and
    \item $l_\alpha[s]=\max\{l : \text{there is an $l$-length loop in the $2n_\alpha$th and $(2n_\alpha+1)$st components\}}$
\end{itemize}
$l_\alpha[s]$ keeps track of the longest loop length found in all components worked on or currently being worked on by $\alpha$ at stage $s$. After defining these three things, take the $w_0$ outcome.

\textbf{Case 2S:} If $\alpha$ has defined $n_\alpha$ (i.e., $n_\alpha[s-1]$ is defined) and is currently challenged by a $P_e$-strategy $\beta$ where $\alpha^\frown\<\infty\>\subseteq\beta$, then $\alpha$ acts as follows. $\beta$ will issue a challenge to $\alpha$ only if $n_\alpha[s-1]>n_\beta$, where $n_\beta$ denotes the $\A$-components on which $\beta$ added diagonalizing loops to. After $\beta$ challenges $\alpha$, $\alpha$ searches for copies of all cycles of length at most $l_\alpha[s]$ in all $\A$-components $2k$, $2k+1$ for $k\leq n_\alpha[s-1]$. Note that $l_\alpha[s]$ now ranges over the diagonalizing loops added by $\beta$ before it issued its challenge. If no copies are found, do not extend $f_\alpha[s]$, leave $n_\alpha$ unchanged, and take the $w_{n_\alpha}$ outcome. Otherwise, extend $f_\alpha[s-1]$ to $f_\alpha$ appropriately, increment $n_\alpha$ by $1$ and take the $\infty$ outcome. In this case, declare $\beta$'s challenge to $\alpha$ has been met.

\textbf{Case 3S:} If $\alpha$ has defined $n_\alpha$ and is not currently challenged by any $P$-strategies, then $\alpha$ continues to search for copies of all cycles of length at most $l_\alpha[s]$ in all $\A$-components $2k$, $2k+1$ for $k\leq n_\alpha[s-1]$. If no copies are found, do not extend $f_\alpha[s-1]$, leave $n_\alpha$ unchanged, and take the $w_{n_\alpha}$ outcome. Otherwise, extend $f_\alpha[s-1]$ appropriately, increment $n_\alpha$ by $1$, and take the $\infty$ outcome.

\subsection{Verification}
We begin with the following lemmas.

\begin{lemma}
    If $f:\A\to\A$ is an embedding of $\A$ into itself, then $f$ is an isomorphism (and is the identity map).
\end{lemma}
\begin{proof}
    Let $f:\A\to\A$ be an embedding. Since embeddings preserve loops and only the root nodes $a_n$ are contained in more than one loop, $f$ must map root nodes to root nodes. We first consider the cases where the $\A$-components have only finitely many cycles attached to each root node.

    In the case where a pair of $\A$-components remain unchanged, that is, they look like the following:
    \begin{align*}
       a_{2n} : 2, 5n+1, \\
       a_{2n+1} : 2, 5n+2;
    \end{align*}

    then it must be the case that $f(a_{2n})=a_{2n}$ and $f(a_{2n+1})=a_{2n+1}$. Otherwise, if only finitely many longer cycles were added to the root nodes, such as in the following configuration for some finite $k>3$:
        \begin{align*}
       a_{2n} : 2, 5n+1, 5n+2, 5n+3,\dots,5n+(k-1),5n+k, \\
       a_{2n+1} : 2, 5n+1, 5n+2, 5n+3,\dots,5n+(k-2),5n+(k+1);
    \end{align*}
    then the cycle of length $5n+k$ is only attached to $a_{2n}$ and the cycle of length $5n+(k+1)$ is only attached to $a_{2n+1}$. This forces $f(a_{2n})=a_{2n}$ and $f(a_{2n+1})=a_{2n+1}$.
    
    We now consider the case where a pair of $\A$-components have infinitely many cycles attached to each root node. That is, $a_{2n}$ and $a_{2n+1}$ are attached to cycles of lengths $5n+k$ for $k\geq 1$. We have that the two components are then isomorphic to each other, and so we have that $f(a_{2n})=a_{2n}$ and $f(a_{2n+1})=a_{2n+1}$. In all three cases, since $\A$ is a directed graph, it must map the loops attached to $a_{2n}$ and to $a_{2n+1}$ identically onto themselves.
\end{proof}

\begin{lemma}\label{lemma: surjective embedding is an isomorphism}
    If $\A\cong\mathcal{M}_i$ where $\mathcal{M}_i$ is a computable directed graph and $f_i:\A\to\mathcal{M}_i$ is an embedding, then $f_i$ is an isomorphism.
\end{lemma}
\begin{proof}
    This follows directly from the previous lemma.
\end{proof}

We now state and prove important observations about the construction.

\begin{lemma}
    An $S_i$-strategy can be challenged by at most one $P_e$-strategy at any given stage.
\end{lemma}
\begin{proof}
    Let $\alpha$ be an $S_i$-strategy and suppose there exists some $P_e$-strategy $\beta$ where $\beta\supseteq\alpha^\frown\<\infty\>$ and $\beta$ challenges $\alpha$ after adding new loops to the $2n^k_\alpha$th and $(2n^k_\alpha+1)$st components of $\A$ and $\B$ (during \textbf{Subcase 2.1P} or \textbf{Subcase 2.3P}). 
    If $\beta$ challenges $\alpha$ at a stage $s$, then $\beta$ takes the $w_{k+1}$ outcome for the first time after also defining a large unused parameter $n^{k+1}_\beta$.
    In particular, $\beta$ will initialize lower priority strategies as it has moved the current true path to the left.
    Strategies extending $\beta^\frown\<w_{k+1}\>$ will now be in \textbf{Subcase 1P} and define new large parameters at stage $s$, and so none will challenge $\alpha$.
    We now check what happens after stage $s$.
    Until $\alpha$ can match the newly added loops in $\A$ to their copies in $\B$ (if any), so it will remain in \textbf{Case 2S}.
    During this time, $\alpha$ will take outcome $w_{n^k_\beta}$.
    Hence, while $\alpha$ remains in \textbf{Case 2S}, all active $P$-strategies $\gamma$ will have $\gamma\supseteq\alpha^\frown\<w_{n^k_\alpha}\>$, and so will not challenge $\alpha$ since $w_{n^k_\alpha}\neq\infty$.
    $\alpha$ only leaves \textbf{Case 2S} once the challenge has been resolved, therefore demonstrating that a second challenge cannot be issued until the first one is resolved, as desired.
\end{proof}

\begin{lemma}\label{lemma: only finitely many permission intervals}
    A $P_e$-strategy $\alpha$ will define at most finitely many valid (permission intervals that remain after a self-initialization as described in \textbf{Case 3P}) permission intervals throughout the construction.
\end{lemma}
\begin{proof}
    Suppose there is some $P_e$-strategy $\alpha$ such that it defines an infinitely long increasing sequence of parameters,
    \[
    n^0 < n^1 <\dots<n^k<\dots,
    \]
    and also an infinitely long increasing sequence of associated $D$-uses which form currently valid permission intervals,
    \[
    u_{n^0}<u_{n^1}<\dots<u_{n^k}<\dots
    \]
    \[
    v_{n^0}<v_{n^1}<\dots<v_{n^k}<\dots
    \]
    $\alpha$ cannot be initialized infinitely many times, otherwise it would only have finitely many stages between each initialization.
    It would then never be able to define an infinite sequence of associated $D$-uses, as each stage can only define at most one $D$-use.
    By this observation, we can assume that all of these $D$-uses are defined after the point where $\alpha$ is never initialized again.
    This means that the above sequences are all computable, as they are defined in order after some non-uniformly determined point in the described computable tree procedure.
    
    Let $m\in\omega$ be arbitrary. Using the sequences listed above, we can computably decide whether $m$ is in $D$ or not.
    Find the least pair $u_{n^k}$ and $v_{n^k}$ such that $u_{n^k}<m<v_{n^k}$.
    We know there must exist such a pair since there are infinitely many listed above in increasing, interleaved order.
    (The $u$ and $v$ sequences are interleaved as a result of the construction described in \textbf{Subcase 2.1P}.)
    Let $s_k$ be the stage at which $\alpha$ sets up this permission interval.
    Note that if $\alpha$ ever enters \textbf{Subcase 2.4P} after $s_k$ it will take the outcome $s$.
    Because $\alpha$ is never initialized at this point, it will continue taking the outcome $s$ forever and stay in \textbf{Case 4P} forever.
    In particular, it will never define new permission intervals.
    Since $\alpha$ defines infinitely many permission intervals, we know that this is not the case.
    In other words, $\alpha$ can never enter \textbf{Subcase 2.4P}.
    By construction, this means that
    \[
    D[s_k]\restriction(u_{n^k},v_{n^k}]=D\restriction(u_{n^k},v_{n^k}],
    \]
    as a change in $D\restriction(u_{n^k},v_{n^k}]$ makes the strategy enter \textbf{Subcase 2.4P}.
    Thus, we can check if $m\in D$ or not by stage $s_k$.
    However, $D$ is a noncomputable c.e.\ set and we have just described a procedure to determine membership in $D$, yielding the desired contradiction.
\end{proof}
 
We draw specific attention to $P_e$-strategies $\alpha$ that take the $\infty$ outcome infinitely often throughout the construction.
They must also define at most finitely many permission intervals before they initialize themselves.
Once they initialize themselves, their previous permission intervals are no longer valid due to the $D$-enumeration, which caused $\alpha$ to take the $\infty$ outcome.
In other words, they act just like any other interval that is initialized infinitely often.

Before we state the next lemma, we introduce the following notation. Suppose $\alpha$ is a $P$-strategy, then let $u_{n^0_\alpha}=u_0$ and let $u_0[s]$ denote the value of $u_0$ at stage $s$.


    

\begin{lemma}\label{lemma: u_0 grows large with every initialization}
    Let $\alpha$ be a $P_e$-strategy such that $u_{n^0_\alpha}=u_0$ is defined at stage $t$ of the construction. Suppose $\alpha$ initializes itself at a stage $t'>t$ by taking the $\infty$ outcome. If $\alpha$ ever redefines $u_0$ at a later stage $t''>t'$, then $u_0[t'']\geq u_0[t]$.
\end{lemma}
\begin{proof}
    Let $\alpha$ and the stages $t<t'<t''$ be as in the hypothesis.
    When $\alpha$ initializes itself at a stage $t'>t$, then the next time $\alpha$ is eligible to act, say at a stage $s_0$, it enters \textbf{Case 2P}. It has no currently defined permission intervals at this point, but has access to its set of parameters, $n^0,n^1,\dots,n^k$, which it defined before it took the $\infty$ outcome at stage $t'$.
    If it can define $u_0$ at stage $t''>t'$, recall that, by \textbf{Subcase 2.1P}, $u_0$ is the maximum of the uses of the $\Phi_e^D[t'']$-computations on the $2n^0_\alpha$th and $(2n^0_\alpha+1)$st components of $\A$. By use convention, because $\Phi_e^D[t'']$ converged again on the aforementioned $\A$-components and $t''>t$, we have that $u_0[t'']\geq u_0[t]$. 
    
    It is worth noting that if the $2n_\alpha^0$th and $(2n_\alpha^0+1)$st components of $\A$ have cycles of lengths $5n_\alpha^0+k$ for $k\geq 3$, if $\Phi_e^D[t'']$ converged on those components again, then $u_0[t'']>u_0[t]$ since the computation now has to account for the longer cycles.
\end{proof}


\begin{lemma}\label{lemma: infinity limsup}
    If $\limsup_s u_0[s]\to\infty$, then $\alpha$ takes the $\infty$ outcome infinitely often. 
\end{lemma}
\begin{proof}
    Suppose that $\alpha$ takes the $\infty$ outcome only finitely often. Without loss of generality, suppose that $\alpha$ is on the true path of the construction (or else $u_0$ is eventually undefined). Let $s_0$ be the least stage such that for all stages $s\geq s_0$, $\alpha$ does not take the $\infty$ outcome at the end of stage $s$. Suppose that by stage $s_1>s_0$ that $\alpha$ has defined the following parameters $n^0,n^1,\dots,n^k$ and permission intervals $(u_{n^i},v_{n^i}]_{0\leq i<k}$. In particular, $u_0[s_1]=u_{n^0}$. After stage $s_0$, when $\alpha$ acts, it either takes some waiting outcome $w_{n^i}$ for $i\geq k$ as in \textbf{Subcases 2.1-3P} or the success outcome if permission is granted as in \textbf{Subcase 2.4P} or if it had previously succeeded as in \textbf{Case 4P}. In each case, we will have that $u_0[s_1]=u_0[s_1+1]$. 
    
    Additionally, since $\alpha$ is on the true path, no higher priority strategy will initialize $\alpha$, and so $u_0[s_1]$ will remain. The only way for $u_0[s_1]$ to be undefined is if there exists a $D$-enumeration smaller than $u_0[s_1]$ as in \textbf{Case 3P} that occurs after stage $s_1$, but if this occurs, then $\alpha$ must take the $\infty$ outcome. Hence, for all stages $t\geq s_1$, $u_0[t]=u_0[s_1]$ and it follows that $\limsup_s u_0[s]=u_0[s_1]$ and this is finite.
\end{proof}

We have by Lemma \ref{lemma: infinity limsup} that $\Phi_e^D$ cannot be total on $\A$ in this case because if $\limsup_s u_0[s]\to\infty$, then $\Phi_e^D$ does not actually converge on the $2n^0$th and $(2n^0+1)$st components of $\A$. So, $\alpha$ satisfies its $P_e$ requirement.

In the case where $\limsup_s u_0[s]\to\infty$, we have the following.

\begin{lemma}\label{fact: some components are infinite now}
    Suppose $\alpha$ takes the $\infty$ outcome infinitely often. If $\alpha$ defines the $D$-uses $u_{n^i}$ for $i\geq 1$ before every stage at which $\alpha$ takes the $\infty$ outcome, then the $2n^0$th and $(2n^0+1)$st components of $\A$ (and of $\B$) will have cycles of length $5n^0+k$ for $k\geq 3$.
\end{lemma}
\begin{proof}
    Recall that for each $i\in\omega$, in order for the associated $D$-use $v_{n^i}$ to be defined, we must first see that $u_{n^i}$ and $u_{n^{i+1}}$ are defined. Once the two associated $D$-uses $u_{n^0}$ and $u_{n^1}$ are defined, then $\alpha$ acts as in \textbf{Case 2.1P} and adds longer cycles to $a_{2n}$ and $a_{2n+1}$, and their corresponding nodes in $\B$. It then defines the value of $v_{n^0}$. Upon self-initialization when it takes the $\infty$ outcome again, $\alpha$ will recover any cycles that disappeared as a result of the small $D$-enumeration as in \textbf{Case 3P}, and so those newly-added longer cycles will reappear in the $2n^0$th and $(2n^0+1)$st components of $\A$ and $\B$. By the lemma's hypothesis, this happens infinitely often, and at each self-initialization, $\alpha$ will add longer cycles to the $\A$- and $\B$-components. Hence, the $2n^0$th and $(2n^0+1)$st components will contain cycles of length $5n^0+k$ for $k\geq 3$.
\end{proof}

For an $\alpha$ which takes the $\infty$ outcome infinitely often but may not be able to define $D$-uses $u_{n^i}$ for $i\geq 1$ ever again before it initializes itself again after some point in the construction, the $2n^0$th and $(2n^0+1)$st components will remain finite. Even though the component stays finite in this case, Lemma \ref{lemma: infinity limsup} still holds. Lemma \ref{fact: some components are infinite now} is an interesting new part of the construction, as now some components in $\A$ possess infinitely many cycles of increasing length. Finding a syntactic description of the automorphism orbits of these components is now more difficult than in previous constructions where all graph components remained finite.

We now state and prove the main verification lemma.

\begin{lemma}[Main Verification Lemma]\label{verification: main verification lemma 1}
    Let $\pi=\liminf_s\pi_s$ be the true path of the construction, where $\pi_s$ denotes the current true path at stage $s$ of the construction. Let $\alpha\subset\pi$.
    \begin{itemize}
        \item[(1)] If $\alpha$ is an $S_i$-strategy, then either $\alpha$ takes outcome $\infty$ infinitely often or there is an outcome $w_n$ and a stage $t$ such that for all $\alpha$-stages $s>t$, $\alpha$ takes outcome $w_n$. If $\A\cong\mathcal{M}_i$, then $\alpha$ takes the $\infty$ outcome infinitely often and $\alpha$ defines a partial embedding $f_\alpha:\A\to\mathcal{M}_i$ which can be extended to a computable isomorphism $\hat{f}_\alpha:\A\to\mathcal{M}_i$.
        \item[(2)] If $\alpha$ is a $P_e$-strategy, then either $\alpha$ takes the outcome $\infty$ infinitely often or is there is an outcome $o\in\{w_n\}_{n\in\omega}\cup\{s\}$ and a stage $t$ such that for all $\alpha$-stages $s>t$, $\alpha$ takes outcome $o$. If $\limsup_s u_0[s]\to\infty$, then $\alpha$ takes the $\infty$ outcome infintely often and so $\Phi_e^D$ cannot be an isomorphism from $\A$ to $\B$.
    \end{itemize}
    In addition, $\alpha$ satisfies its assigned requirement.
\end{lemma}
\begin{proof}
    We first prove $(1)$. Let $\alpha\subseteq\pi$ be an $S_i$-strategy and let $s_0$ be the least stage such that for all $s\geq s_0$, $\alpha\leq_L\pi_s$. Suppose that $\alpha$ only takes the $\infty$ outcome finitely often. Fix an $\alpha$-stage $s_1>s_0$ such that $\alpha$ does not take the $\infty$ outcome at any $\alpha$-stage $s\geq s_1$. Suppose that $\alpha$ takes the $w_n$ outcome at stage $s_1$. There are two cases to consider.

    If $\alpha$ is not challenged at stage $s_1$, then $\alpha$ acts as in \textbf{Case 3S} of its strategy. $\alpha$ cannot be challenged by any $P$-strategies extending $\alpha^\frown\<w_n\>$ and so it will remain in \textbf{Case 3S} taking outcome $w_n$ for all future $\alpha$-stages until it finds copies of the $\A$-components corresponding to $w_n$. However, if this happens, $\alpha$ would take the $\infty$ outcome after stage $s_1$, and so $\alpha$ must never find these copies. Thus, $\alpha$ takes the $w_n$ outcome at every $\alpha$-stage $s\geq s_1$ and $\mathcal{M}_i$ does not contain the copies of the $\A$-components and thus cannot be isomorphic to $\A$.

    If $\alpha$ is challenged at stage $s_1$, then $\alpha$ acts as in \textbf{Case 2S} of its strategy.
    By the same argument as the one above, $\alpha$ can never meet this challenge by finding images for the new loops in the $\A$-components corresponding to $w_n$. It follows that $\mathcal{M}_i\not\cong\A$ and that $\alpha$ takes the $w_n$ outcome for all $\alpha$-stages $s\geq s_1$.

    These arguments show that if $\A\cong\mathcal{M}_i$, then $\alpha$ takes the $\infty$ outcome infinitely often.
    Furthermore, if $\A\not\cong\mathcal{M}_i$ then $\alpha$ stays on one $w_n$ outcome after some stage.

    We prove one more fact about the behavior of $S_i$-strategies in the case where $\A\cong\mathcal{M}_i$.
    Let $n_\alpha[s]$ denote the value of $n_\alpha$ at the end of stage $s$.
    We now show that $n_\alpha[s]\to\infty$ as $s\to\infty$ in the case where $\alpha$ takes the $\infty$ outcome infinitely often because $\A\cong\mathcal{M}_i$. Suppose $\alpha$ takes the $\infty$ outcome at a stage $s>s_0$. We then have that $n_\alpha[s]=n_\alpha[s-1]+1$ by the action in either \textbf{Case 2S} or \textbf{Case 3S}. Because $\A\cong\mathcal{M}_i$, if $\alpha$ is issued a challenge by some $P$-strategy $\beta\supseteq\alpha^\frown\<\infty\>$, there will always be a stage after the challenge has been issued where $\alpha$ can extend its map appropriately and takes the $\infty$ outcome. Every time this occurs, $n_\alpha$ is incremented by $1$, and so $n_\alpha[s]\to\infty$ as $s\to\infty$.

    We now prove $(2)$, so proceed by assuming that $\alpha$ is a $P_e$-strategy.
    Since $\alpha\subset\pi$, fix $s_0$ to be the least stage such that for all stages $s\geq s_0$, $\alpha<_L\pi_s$. Suppose that $\alpha$ takes the $\infty$ outcome only finitely often. Let $s_1\geq s_0$ be the least stage such that for all stages $s\geq s_1$, $\alpha$ does not take the $\infty$ outcome at the end of stage $s$.

    By Lemma \ref{lemma: infinity limsup}, we have that if $\limsup_s u_0[s]\to\infty$, then $\alpha$ takes the $\infty$ outcome infinitely often. We now show that if $\alpha$ takes the $\infty$ outcome infinitely often, then $\Phi_e^D$ cannot be an isomorphism and so $\alpha$ meets its requirement in this way. 

    Assume towards contradiction that $\Phi_e^D$ is an isomorphism. Let $s'$ be a large enough stage after stage $s_0$ such that $\alpha$ has defined the following parameters
    \[
    n^0_\alpha<n^1_\alpha<\dots<n^k_\alpha,
    \]
    permission intervals $(u_{n^i_\alpha},v_{n^i_\alpha}]_{0\leq i<k}$, and that if $\Phi_e^D[s']\downarrow$ correctly on the $2n^i_\alpha$th and $(2n^i_\alpha+1)$st components of $\A$, then $\Phi_e^D[s']=\Phi_e^D$ on these two components with true $D$-use $u_{n^i_\alpha}[s']=u_{n^i_\alpha}$. But if $\alpha$ takes the $\infty$ outcome infinitely often after stage $s_0$, there is a stage $t>s'$ such that $\alpha$ takes the $\infty$ outcome at the end of stage $t$. But this would contradict the fact that $s'$ was chosen large enough so that
    \[
    D[s']\restriction u_{n^i_\alpha}=D\restriction u_{n^i_\alpha}
    \]
    since a number smaller than $u_{n^i_\alpha}$ entered $D$ at stage $t$ as described in \textbf{Case 3P}.

    In the case where $\alpha$ does not take the $\infty$ outcome infinitely often, we break into the following subcases.
    Again, let $s_1>s_0$ be the least stage such that $\alpha$ does not take the $\infty$ outcome after stage $s_1$.
    If $\alpha$ remains in the first part of \textbf{Subcase 2.1P} or \textbf{Subcase 2.3P}, or in the waiting part of \textbf{Subcase 2.2P} after setting up the permission interval, then $\alpha$ will continue to take the corresponding waiting outcome in all those subcases.
    This is because the path will move left when moving away from any of the above parts, but as $s_1>s_0$, we know that this will not happen.
    
    Thus, $\alpha$ will take some waiting outcome, e.g. $w_k$ if it remains waiting for $\Phi_e^D$ to converge correctly on the $2n^k$th and $(2n^k+1)$st components of $\A$ like in the first part of \textbf{Subcase 2.1P}, at all $\alpha$-stages $s\geq s_1$. In this case, the $P_e$ requirement is satisfied trivially since $\Phi_e^D$ cannot be an isomorphism between $\A$ and $\B$.
    The other above-mentioned cases are similar.

    Otherwise, suppose there is some $\alpha$-stage $s>s_1$ where, after waiting, $\alpha$ is able to set up a permission interval for $n^k$ as in \textbf{Subcase 2.2P} of the $P_e$-strategy. It then takes outcome $w_{k+1}$ after defining a new large $\alpha$-parameter $n^{k+1}$.
    We have by Lemma \ref{lemma: only finitely many permission intervals} that it is not the case that $\alpha$ will choose outcomes $w_l$ where $l\to\infty$ as $s\to\infty$ for $s\geq s_1$. So, either $\alpha$ remains in some waiting outcome for all stages $s\geq s_1$ like in the first case above, or there exists some $j\leq k$ such that
    \[
    D[s']\restriction(u_{n^j},v_{n^j}]\neq D[s_1]\restriction(u_{n^j},v_{n^j}]
    \]
    for some $\alpha$-stage $s'>s_1$. 
    In this case, we can switch the diagonalizing loops on the $2n^j$th and $(2n^j+1)$st components of $\B$ and carry out the rest of the actions described in \textbf{Subcase 2.4P}.
    $\alpha$ then takes the success outcome $s$ at stage $s'$. 
    
    Additionally, since $\alpha\subseteq\pi$ and $\alpha$ does not take the $\infty$ outcome after stage $s_1$, $\alpha$ continues to take the $s$ outcome for all remaining $\alpha$-stages after stage $s'$ and that $D[s']\restriction s'=D\restriction s'$ (and so the diagonalizing loops stay in their new positions for the rest of the construction). Thus we have successfully diagonalized against $\Phi_e^D$ with the $2n^j$th and $(2n^j+1)$st components of $\A$, satisfying the $P_e$ requirement.\end{proof}

\section{Theorem 1.2}

We now show that the other combination, where the structure we build is not computably categorical but is categorical relative to some given noncomputable c.e\ set, is also possible.

\begin{theorem}
    Given a noncomputable c.e. set $D$, there exists a computable directed graph $\mathcal{A}$ that is computably categorical relative to $D$ but is not computably categorical.
\end{theorem}

\subsection{Requirements and Tree of Strategies} We list our requirements that we seek to satisfy.
\[\text{$R_e$: $\Phi_e:\mathcal{A}\to\mathcal{B}$ is not an isomorphism; and}\] 
\[\text{$S_i$: If $M_i^D\cong \mathcal{A}$, then $M_i^{D}$ is $D$-isomorphic to $\mathcal{A}$.}\]

Here, $\mathcal{B}$ is a fixed computable copy of $\mathcal{A}$ that we build alongside the latter.

Strategies will act to fulfill these requirements along a tree.
The tree will be built like the one from the previous section.
The outcomes will have similar intuitive meanings as before (though diagonalizing strategies will no longer need to use the $\infty$ outcome as the $P$-strategies from before did).
The basic mechanics and injury mechanism along the tree are standard and the same as before.

\subsection{Formal strategies}

   \subsubsection{$R_e$-strategies}

    Let $D$ be our noncomputable c.e.\ set. Again we assume that at most one element is enumerated into $D$ at each stage $s$. Let $d_s$ denote the last number enumerated into $D$ at the beginning of stage $s$, i.e. we define $d_s=x$ where $x\in D[s]$ but for all $t<s$, $x\not\in D[t]$, if such an $x$ exists. If no such $x$ exists, then let $d_s=d_{s-1}$.

    Every requirement is assigned to a strategy that acts to fulfill the requirement.
    The first time a strategy acts is called the \textit{initialization}.
    Each subsequent time, we are \textit{revisiting} the strategy.
    Each time a strategy acts, it will go through a series of steps.
    Some of these steps will only happen on the initialization or upon revisiting the strategy.
    The exact action will often depend on what the strategy did previously and how other strategies have acted before the given stage.
    We now describe the action of the strategies aiming to fulfill the $R_e$ requirements.

    \textbf{Step 1:} Upon initialization at a stage $s$, a parameter $n^e_0$ is chosen to be larger than any parameter picked in the construction so far. For this description, we let $n=n^e_0$. The following components are then created in $\mathcal{A}[s]$ and in $\mathcal{B}[s]$ (See Figure 4):
            \begin{align*}
        a_{2n} : 2, \textcolor{red}{5n+1} & \ \ \ \ a_{2n+1} : 2, \textcolor{blue}{5n+2} \\
        \textcolor{darkgray}{b_{2n}}: 2, \textcolor{red}{5n+1} & \ \ \ \ \textcolor{darkgray}{b_{2n+1}} : 2, \textcolor{blue}{5n+2} .
    \end{align*}

    \begin{figure}[h]
   \centering
   \makebox[\textwidth][c]{\includegraphics[width=.7\textwidth]{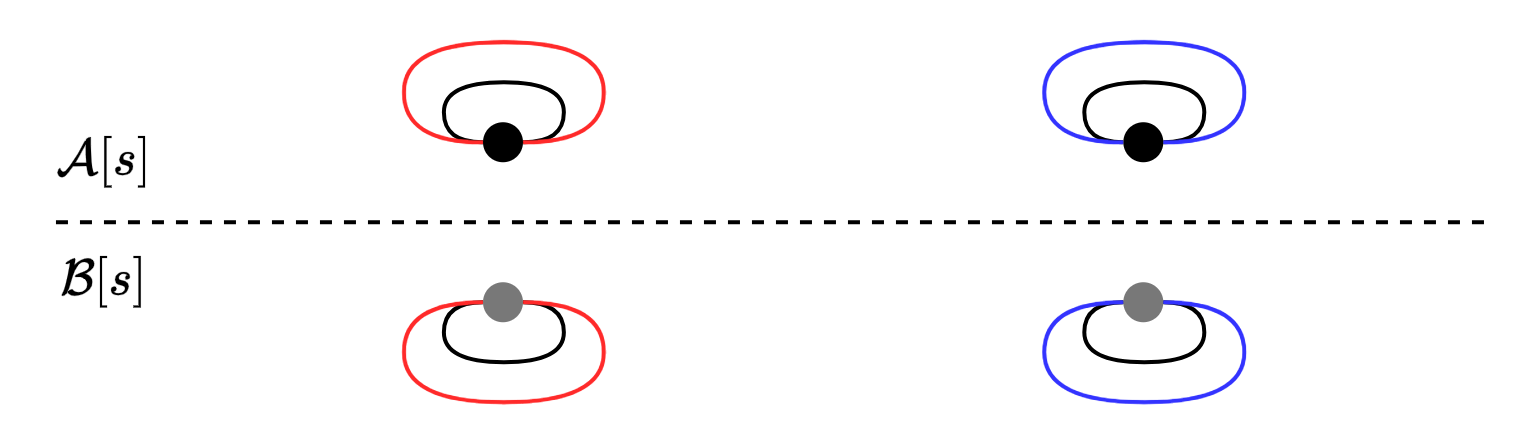}}
    \caption{}
    \end{figure}
    
    \textbf{Step 2:} 
    To describe this step, we will refer to parameters of the form $m_i[s']$.
    This is a parameter associated with the strategy acting to fulfill an $S_i$ requirement.
    We therefore do not precisely describe how these are calculated until we discuss the $S_i$ strategies in more detail.
    For now, it is enough to know that $m_i[s']$ gives the number of components in the domain of the partial isomorphism $f_i[s']$ between $M_i^D$ and $\mathcal{A}$ constructed by $S_i$ at stage $s'$.
    
    Upon being revisited at a stage $s'>s$, $R_e$ asks the following questions:
    \begin{itemize}
        \item[(i)] Which indices $i$ have the property that $m_i[s']>n$? Note that if $m_i[s']>n$, the map $f_i[s']$ exists on the two $\mathcal{A}[s']$-components above.
        \item[(ii)] Does $\Phi_e$ map the two $\mathcal{A}[s']$-components above correctly into $\mathcal{B}[s']$?
    \end{itemize}
    If there exist such indices $i$ in (i), then let $u^i_n[s']$ denote the use of the map $f_i[s']$ on the aforementioned $\mathcal{A}[s']$-components. If the answer is yes to (ii), we proceed to \textbf{Step 3}. Otherwise, take the $w_n$ outcome and at the next time we revisit the strategy, we will return to \textbf{Step 2} and ask the same question (ii) until the answer is yes (if ever) before proceeding. While we wait to move on to \textbf{Step 3} for this parameter $n$, we call it the \textit{current active parameter} of $R_e$ and continue taking the $w_n$ outcome. Its corresponding pair of $\A$-components, which we denote by $A_n$, are the \textit{current active components} of $R_e$.
    
    \textbf{Step 3:} We break into the following cases. We will suppress the stage notation throughout this description unless otherwise needed.
    
    \textbf{Case 3.1}: Suppose there is no $i$ such that $m_i[s']>n$. Then we add new cycles in $\mathcal{A}$ and in $\mathcal{B}$ to diagonalize against $\Phi_e$ with the following configuration (See Figure 5):
    \begin{align*}
        a_{2n} : 2, 5n+1, 5n+2, \textcolor{red}{5n+3} & \ \ \ \ a_{2n+1} : 2, 5n+1, 5n+2, \textcolor{blue}{5n+4} \\
        \textcolor{darkgray}{b_{2n}} : 2, 5n+1, 5n+2, \textcolor{blue}{5n+4} & \ \ \ \ \textcolor{darkgray}{b_{2n+1}} : 2, 5n+1, 5n+2, \textcolor{red}{5n+3}.
    \end{align*}

   \begin{figure}[h]
   \centering
   \makebox[\textwidth][c]{\includegraphics[width=.7\textwidth]{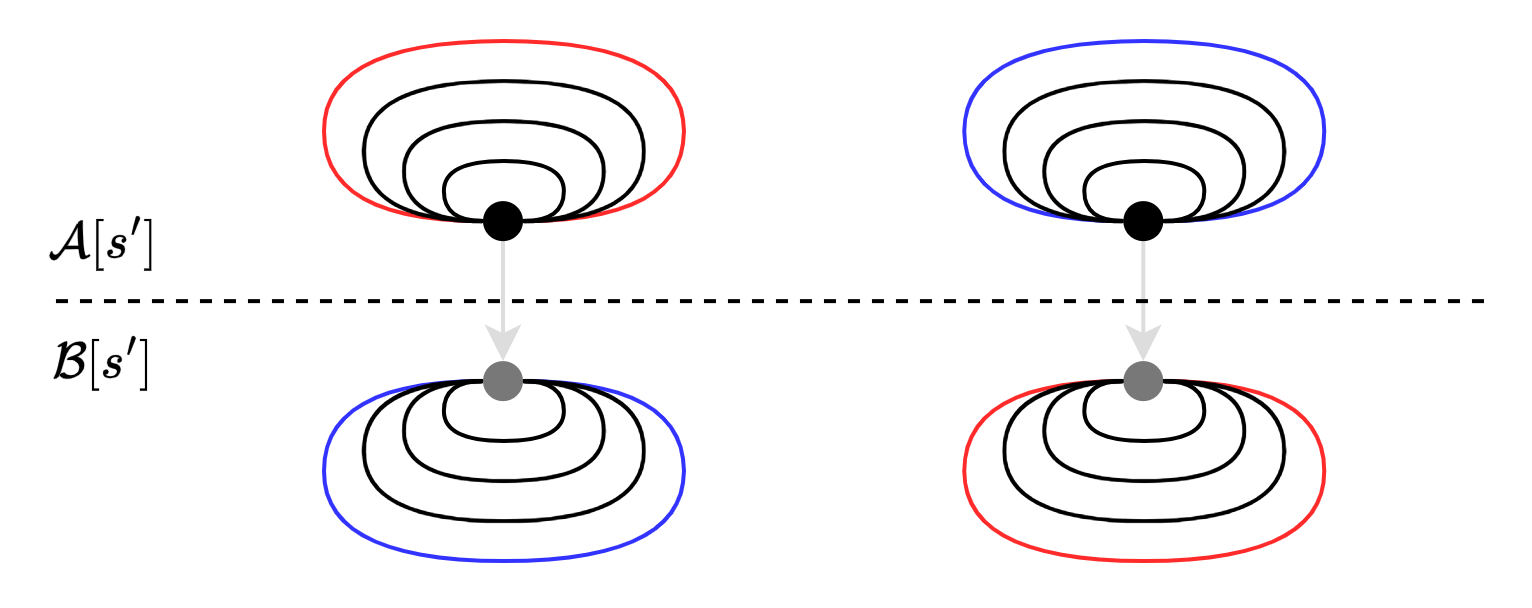}}
    \caption{}
    \end{figure}
    
    $R_e$ needs to take no further action and can end its strategy in this case by taking the $s$ outcome.
    
    \textbf{Case 3.2}: Suppose there exist indices $i$ such that $m_i[s']>n$ and so the $D$-parameters $u^i_n$ are defined. Then, homogenize the components in $A_n$ in the following way (See Figure 6):
    \begin{align*}
        a_{2n} : 2, 5n+1, 5n+2 & \ \ \ \ a_{2n+1} : 2, 5n+1, 5n+2 \\
        \textcolor{darkgray}{b_{2n}} : 2, 5n+1, 5n+2 & \ \ \ \ \textcolor{darkgray}{b_{2n+1}} : 2, 5n+1, 5n+2,
    \end{align*}
    \begin{figure}[h]
   \centering
   \makebox[\textwidth][c]{\includegraphics[width=.7\textwidth]{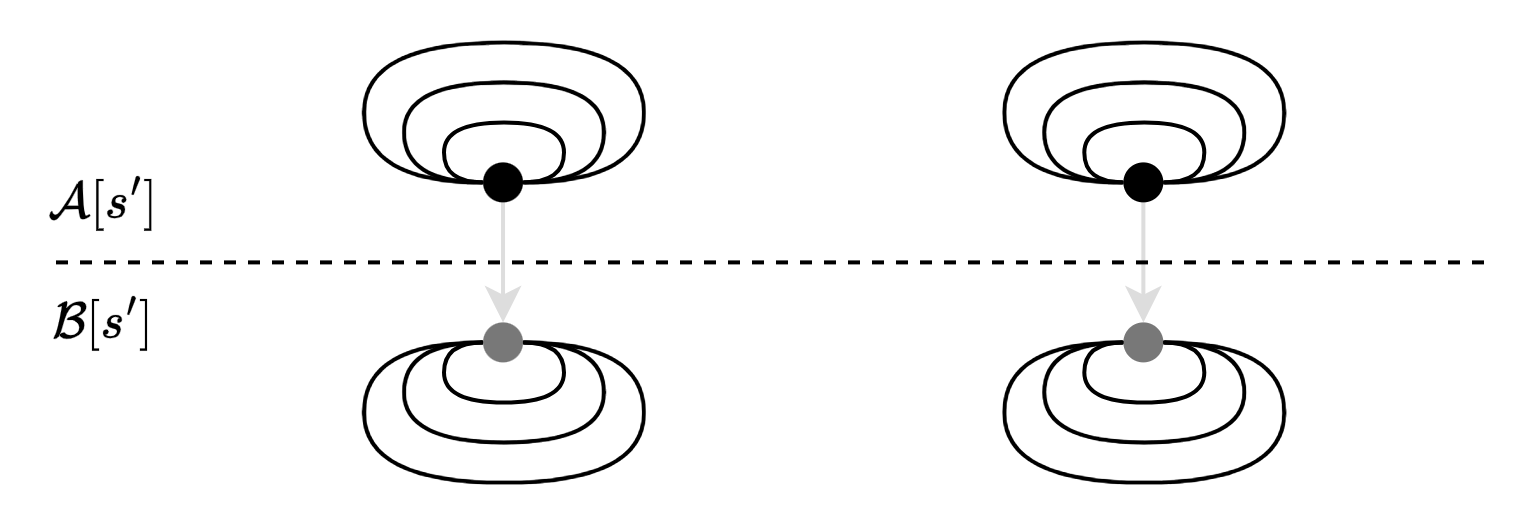}}
    \caption{}
    \end{figure}
    
    \noindent and check if $d_{s'}<\min\{u^i_n : \text{$i$ is such that $m_i[s']>n$}\}$ (we refer to this property of $d_{s'}$ as $d_{s'}$ \textit{being small enough}). If $d_{s'}$ is small enough, then we add additional cycles to the components in $A_n$ to diagonalize against $\Phi_e$ as follows:
        \begin{align*}
        a_{2n} : 2, 5n+1, 5n+2, \textcolor{red}{5n+3} & \ \ \ \ a_{2n+1} : 2, 5n+1, 5n+2, \textcolor{blue}{5n+4} \\
        \textcolor{darkgray}{b_{2n}} : 2, 5n+1, 5n+2, \textcolor{blue}{5n+4} & \ \ \ \ \textcolor{darkgray}{b_{2n+1}} : 2, 5n+1, 5n+2, \textcolor{red}{5n+3}.
    \end{align*}

    After adding these new cycles, for all $S_i$-strategies $\beta$ where $i$ is such that $m_i[s']>n$ and $\alpha\supseteq\beta^\frown\<\infty\>$, $R_e$ redefine $m_i[s'+1]=n$. $R_e$ does not need to take any further action and thus can take the $s$ outcome. Otherwise, if $d_{s'}>u^i_n$ for some $i$, $R_e$ defines a new large $n^e_1$, and begins taking the previous steps with this parameter in mind, and lets the next strategy act after taking the $w_{n^e_1}$ outcome.
    To be explicit, the strategy will create new components for this $n^e_1$ as in \textbf{Step 1} and subsequently try to add these components to the list of current active components as described in \textbf{Step 2} when the strategy is revisited. 
    

    \textbf{Step 4:} Suppose $R_e$ is revisited at stage $s''>s'$ and its currently active components are in $A_k$ for some $k\geq 0$. In order of creation, it revisits all of its previously created components $A_0,\dots,A_{k-1}$ to see if $D$-permission has been granted (i.e., there is an enumeration into $D$ that is small enough) or not for each one. We have the following cases.

    \textbf{Case 4.1:} Suppose there is a least $j<k$ such that $d_{s''}<\min\limits_i\{u^i_j\}$, then we say that $R_e$ has \textit{received permission from} $D$ to finish its diagonalization using the components in $A_j$. $R_e$ now adds the new cycles of lengths $5j+3$ and $5j+4$ in the following way:
        \begin{align*}
        a_{2j} : 2, 5j+1, 5j+2, \textcolor{red}{5j+3} & \ \ \ \ a_{2j+1} : 2, 5j+1, 5j+2, \textcolor{blue}{5j+4} \\
        \textcolor{darkgray}{b_{2j}} : 2, 5j+1, 5j+2, \textcolor{blue}{5j+4} & \ \ \ \ \textcolor{darkgray}{b_{2j+1}} : 2, 5j+1, 5j+2, \textcolor{red}{5j+3}.
    \end{align*}
    $R_e$ then redefines $m_i[s']=j$ for all applicable $S_i$-strategies, halts its work for all of its other parameters defined so far, and takes the $s$ outcome. We let the next strategy act.

    \textbf{Case 4.2:} If for all $j<k$, $R_e$ does not receive permission from $D$ to finish its diagonalization using any of the $A_j$'s, then we continue the $R_e$-strategy for the components in $A_k$ by checking if it can proceed to \textbf{Step 3}. Continue taking the $w_k$ outcome.

    Note that in \textbf{Case 4.1}, we let $R_e$ finish its diagonalization using $A_j$ for the least $j$ which received $D$-permission. For all $S_i$-strategies such that the uses $u^i_n$ are defined for $n\geq j$, the corresponding maps also get deleted because $u^i_n\geq\min\limits_i{u^i_j}$ by definition. Later during verification, in the case that $\mathcal{A}\cong\mathcal{M}^D_i$, whenever $S_i$ recovers its maps on the components in $A_j$, we will show that it will also eventually get the opportunity to recover its maps on the later components found in $A_n$ for $n\geq j$.

   \subsubsection{$S_i$-strategies.} We let $m_i[s]$ denote the value of the $S_i$-parameter at the beginning of stage $s$ in the construction, where $m_i[s]=n$ denotes that the $S_i$-strategy is currently looking for copies of the components with root nodes $a_{2n}$ and $a_{2n+1}$ (equivalently, the components found in $A_n$). Throughout this description, when we are looking for copies of $\A$-components in $\mathcal{M}_i^D$, we mean that we will always look for the oldest (the copy which appeared first and remains throughout the stages) and least indexed (if there are multiple copies at some stage, pick the least indexed one) copies which show up throughout the enumeration of $\mathcal{M}_i^D$. 

\textbf{Step 1:} Upon initialization at stage $s$, we define our parameter $m_i=0$ and $f_i:\mathcal{A}\to\mathcal{M}^D_i$ to be the empty map. Take the $w_0$ outcome.
    
\textbf{Step 2:} Upon being revisited at a stage $s'>s$, we first check if $m_i[s']$ has been redefined by some $R_e$-strategy at some previous stage before stage $s'$ (see either \textbf{Case 3.2} or \textbf{Case 4.1} of the $R_e$-strategy for more details). For $m_i[s']$, we search for copies of the corresponding $\mathcal{A}$-components in $\mathcal{M}_i^D[s']$. If copies are found, we extend our definition of $f_i[s'-1]$ accordingly to obtain $f_i[s']$ and define a new large $D$-use $u^i_{m_i[s']}$ for it. Then, we increment $m_i[s']$ by 1, take the $\infty$ outcome, and let the next strategy act. Otherwise, we let the next strategy act after taking the $w_{m_i[s']}$ outcome and continue searching once $S_i$ is revisited at a later stage $s''>s'$.

Note that sometimes our current guess about the components of $\M_i$ may be incorrect, so we may increase $m_i$ more than we intuitively had any right to based on faulty information.
As it turns out, this is not an issue.
The key input to the verification will be that $m_i$ acts as we expect when $\M_i$ and $\A$ are actually isomorphic.
In other cases, it will not matter what $m_i$ does.

\subsection{Verification}

We now verify that our construction in the previous section works. We begin with some helpful lemmas regarding our construction.

\begin{lemma}\label{lemma: permission will eventually be granted}
    For all $e\in\omega$, $R_e$ only defines finitely many $\mathcal{A}$-components throughout the construction.
\end{lemma}

\begin{proof}
    Let $e\in\omega$ and suppose towards contradiction that $R_e$ defines infinitely many parameters $n_0<n_1<\dots<n_k<\dots$ and thus infinitely many pairs of $\A$-components $\{A_{n_i}\}_{i\in\omega}$. No parameters $n_i$ ever reaches \textbf{Case 3.1}, as that would imply that the strategy will take no further action afterwards, therefore each parameter always enters \textbf{Case 3.2}. Let $n_k$ be a parameter which enters \textbf{Case 3.2} for the first time at stage $s_0$. Since infinitely many parameters greater than $n_k$ are defined, then there exists a stage $s_1>s_0$ such that $d_{s_1}>u^i_{n_k}$ for some $i$. Additionally, for all stages $t\geq s_1$, we have that $d_t$ never becomes smaller than $\min\{u^i_{n_k} : \text{$i$ such that $m_i[t]>n_k$}\}$ or else $R_e$ would enter \textbf{Case 4.1} for $n_k$ and stop acting. So we have for each $n_k$, there exists a stage $s_k$ such that for all stages $t\geq s_k$, $d_t>\min\{u^i_{n_k} : \text{$i$ such that $m_i[t]>n_k$}\}$ and as the $n_k$ parameters grow larger so do the $D$-uses $u_{n_k}$ associated with them. Using this, we can compute longer and longer initial segments of $D$, and so $D$ is computable, yielding a contradiction. \end{proof}
    


\begin{lemma}\label{lemma: diagonalization is successful}
    Let $e\in\omega$ and suppose $R_e$ has some parameter $n^e_k$ for which the strategy enters either \textbf{Case 3.1},  \textbf{Case 3.2} and is granted immediate $D$-permission, \textbf{Case 4.1}. Then, $R_e$ successfully diagonalizes against the computable map $\Phi_e$.
\end{lemma}
\begin{proof}
    Let $R_e$ be a strategy and suppose that for the parameter $n^e_k=n$, the strategy enters \textbf{Case 3.1} at a stage $s$ in the construction. This means that the strategy is now able to add the diagonalizing loops to the components in $A_n$ and to their respective $\B$-components.
    We claim that, after this, no other changes are made to $A_n$. No other $R$-strategy can change $A_n$ because, by definition, their set of parameters and thus components is disjoint from $R_e$'s set of parameters and components. $R_e$ also takes no further action after adding the diagonalizing loops to $A_n$. 
    $S$-strategies never change the components of $\A$.
    Taken together, $A_n$ remains unchanged after the diagonalizing loops are added.
    In particular, $a_{2n}$ is in a cycle of length $5n+3$ and not one of length $5n+4$, but $\Phi_e(a_{2n})$ is in a cycle of length $5n+4$ and not one of length $5n+3$.
    Therefore, $\Phi_e$ is not an isomorphism, as desired.

    Now suppose that $R_e$ with parameter $n$ enters \textbf{Case 3.2} and is granted immediate $D$-permission or it enters \textbf{Case 4.1} at a stage $s$ in the construction. As entering case \textbf{Case 4.1} requires an appropriate $D$-permission so our verification in both cases are analogous. $R_e$ now completes its diagonalization against $\Phi_e$ using cycles of lengths $5n+3$ and $5n+4$. We similarly have that $a_{2n}$ is in a cycle of length $5n+3$ but $\Phi_e(a_{2n})$ is not. Just the same as above, no other $R$-strategy or $S$-strategy will change the configuration of loops in $A_n$ and $R_e$ takes no further action after it adds the new cycles above. Thus, $R_e$ succeeds in this case as well.
    Note that in all outcomes, we still have that $\mathcal{A}\cong\mathcal{B}$. \end{proof}

\begin{lemma}\label{lemma: if graphs are isomorphic, then f is defined on all of A}
    Let $i\in\omega$. If $\mathcal{A}\cong\mathcal{M}_i^D$, then $m_i[s]\to\infty$ as $s\to\infty$.
\end{lemma}

\begin{proof}
    Suppose $\mathcal{A}\cong\mathcal{M}_i^D$, so each connected component in $\A$ has a copy in $\mathcal{M}_i^D$. Let $m_i$ be $S_i$'s parameter. If $m_i$ is never redefined to equal some $R_e$'s parameter, then because the two graphs are isomorphic, eventually the strategy finds corresponding components extends the isomorphism and increments $m_i$.
    Thus, $m_i$ continues to increase and never descreases and in this case we are done.
    So suppose that at some stage $s$ in the construction, there is some $R_e$ requirement with its active parameter $n^e=n$ that ends up in \textbf{Case 3.2} and thus $m_i[s]$ is redefined to equal $n$, and so in particular, $m_i[s]<m_i[s-1]$. In order to prove this lemma, it suffices to show that for each $m_i$ value like this, there are only finitely many stages $t$ after stage $s$ such that $m_i[t]<m_i[s-1]$ still. We will eventually see a stage $s'>s$ such that $m_i[s']>m_i[s]$ because $\mathcal{A}\cong\mathcal{M}_i^D$, and so there exists a copy of $A_n$ with the new loops added by the $R_e$-strategy. Once the copies of $A_n$ have been found, $S_i$ will increment $m_i$ until eventually it becomes the same value of $m_i[s-1]$.

    By definition, $m_i$ only decreases once a $R_e$ strategy adds the diagonalizing loops to one of its finitely many components that had previously been mapped by $S_i$. Moreover, it need be the case that such an $R_e$-strategy extends $S_i$ on its infinite outcome. Because $A\cong\M_i^D$, there is a stage after $R_e$ issues its challenge where $S_i$ will recover on its redefined $m_i$ value and thus can take the infinite outcome again. Before this stage, all $R$-strategies extending $S_i$ on a finite outcome cannot challenge and redefine $S_i$'s parameter. Additionally, once $S_i$ recovers and takes the infinite outcome, if it is challenged again by a lower priority $R$-strategy, this $R$-strategy will redefine $m_i$ to be some large parameter distinct from the parameter belonging to $R_e$ by the construction, and so we have that $m_i[s]\to\infty$ as $s\to\infty$.
\end{proof}

\begin{lemma}\label{lemma: if graphs are isomorphic, then f is correct}
    Let $i\in\omega$. If $\mathcal{A}\cong\mathcal{M}_i^D$, then $f_i$ is an isomorphism between $\mathcal{A}$ and $\mathcal{M}_i^D$.
\end{lemma}
\begin{proof}
    By Lemma \ref{lemma: if graphs are isomorphic, then f is defined on all of A}, $f_i$ is defined on all of $\mathcal{A}$ if $\A\cong\mathcal{M}_i^D$. All that remains now is to show that $f_i$ is correct. If an $\mathcal{A}$-component is never changed by some $R_e$ strategy (i.e., it never gained the loops used to diagonalize against some computable map), then $f_i$ will be correct on those components by our construction.

    Suppose for some $n$, $A_n$ was changed by some $R_e$ strategy. We break into the following cases. 
    
    \textbf{Case A:} Suppose $R_e$ adds the cycles of length $5n+3$ and $5n+4$ to $A_n$ as in \textbf{Case 3.1} at stage $s$ of the construction. Since $R_e$ ended up in \textbf{Case 3.1} for $n$, this means that $m_i[s]<n$. Therefore, $S_i$ had not yet mapped the loops of the $A_n$ component, so it must now only wait until there is a stage $s'$ such that $m_i[s']=n-1$ before working on the $A_n$ component. Because $\A\cong\mathcal{M}_i^D$, and by the previous lemma, eventually we will have a later stage $s''>s'$ such that $m_i[s'']=n$ , and so $f_i$ will eventually be defined correctly on all cycles of length $5n+l$ for $l\in\{1,2,3,4\}$ in $A_n$.

    \textbf{Case B:} Suppose $R_e$ added the cycles of length $5n+3$ and $5n+4$ to $A_n$ as in \textbf{Case 3.2} at stage $s$ of the construction. By definition, this occurs only once a $D$-enumeration happens below $\min\{u^i_n : \text{$i$ is such that $m_i[t]>n$}\}$ at some stage $t$. As each $S_k$ strategy assigns uses to its $f_k$ maps, this means that the definition of $f_k$ on $A_n$ disappears for every $S_k$ with $m_k[t]>n$. If $S_i$ is one of said strategies, it will simply be able to redefine $f_i$ on $A_n$ at a later stage and when it does, it will include the newly added cycles of lengths $5n+3$ and $5n+4$. Otherwise, $S_i$ will simply wait until $m[s']=n-1$ at some later stage $s'$ to create the map.
    
    If $D$-permission is never granted for for $A_n$, then recall that $R_e$ homogenized the components in $A_n$ at the beginning of \textbf{Case 3.2}. Although we have the original definition of $f_i$ on $A_n$, since $\A\cong\mathcal{M}_i^D$, copies of the homogenized $A_n$ will eventually appear and stay in $\mathcal{M}_i^D$. Additionally, since the components are both isomorphic, the original definition of $f_i$ can be easily extended, with the help of $D$, on the new loops to form a correct map defined on all nodes and edges found in $A_n$.
\end{proof}

We now state and prove the main verification lemma for this construction.

\begin{lemma}[Main Verification Lemma]\label{verification: main verification lemma 2}
    Let $\pi=\liminf_s\pi_s$ be the true path of the construction, where $\pi_s$ denotes the current true path at stage $s$ of the construction. Let $\alpha\subset\pi$.
    \begin{itemize}
        \item[(1)] If $\alpha$ is an $S_i$-strategy, then either $\alpha$ takes outcome $\infty$ infinitely often or there is an outcome $w_n$ and a stage $t$ such that for all $\alpha$-stages $s>t$, $\alpha$ takes outcome $w_n$. If $\A\cong\mathcal{M}_i^D$, then $\alpha$ takes the $\infty$ outcome infinitely often and $\alpha$ defines a $D$-computable isomorphism $f_\alpha^D:\A\to\mathcal{M}_i^D$.
        \item[(2)] If $\alpha$ is an $R_e$-strategy, then there is an outcome $o\in\{w_n\}_{n\in\omega}\cup\{s\}$ and a stage $t$ such that for all $\alpha$-stages $s>t$, $\alpha$ takes outcome $o$.
    \end{itemize}
    In addition, $\alpha$ satisfies its assigned requirement.
\end{lemma}
\begin{proof}
    Towards (1), let $\alpha\subset\pi$ be an $S_i$-strategy and let $s_0$ be the least stage such that for all $s\geq s_0$, $\alpha\leq_L\pi_s$. If $\alpha$ only takes the $\infty$ outcome finitely often, then the argument is the same as in the first part of the proof of Lemma \ref{verification: main verification lemma 1}(1). If $\A\cong\M_i^D$, then by Lemma \ref{lemma: if graphs are isomorphic, then f is defined on all of A} we have that we can build a $D$-computable embedding on all of $\A$. By Lemma \ref{lemma: if graphs are isomorphic, then f is correct}, such an embedding will be the needed $D$-computable isomorphism between $\A$ and $\M_i^D$ and so $\alpha$ satisfies its assigned requirement.

    To show (2), we have by Lemma \ref{lemma: diagonalization is successful}, if $R_e$ enters either \textbf{Case 3.1}, \textbf{Case 3.2} with $D$-permission granted immediately, or \textbf{Case 4.1}, then it will meet its requirement by directly diagonalizing against $\Phi_e$. Otherwise, we have that by Lemma \ref{lemma: permission will eventually be granted}, $R_e$ will only define finitely many $\A$-components throughout the construction, call them $A_0,A_1,\dots,A_n$. If none of these components enter either \textbf{Case 3.1} or \textbf{Case 4.1}, then for each $A_n$, there is a stage $s$ in the construction such that $R_e$ cannot proceed to \textbf{Step 3} for $A_n$. This means that $\Phi_e$ never maps the components of $\A$ to the components of $\B$ correctly.  So, $\Phi_e$ cannot be an isomorphism between $\A$ and $\B$. $R_e$ meets its requirement trivially in this case. In other words, out of the finitely many components created by the strategy, either a component succeeds in adding the diagonalizing loops or the strategy wins by vacuity on the existence of the isomorphism it is trying to diagonalize against.
\end{proof}

\section{Open Questions and Future Work}
One way to push our results further would be to extend the degrees which we are considering.
Let $\M$ be a computable structure.

\begin{question}
    Are there any (non-c.e.) degrees $\mathbf{d}$ where $\M$ being computably categorical relative to $\mathbf{d}$ implies that $\M$ is computably categorical?
\end{question}

\begin{question}
    Are there any (non-c.e.) degrees $\mathbf{d}$ where $\M$ being computably categorical implies that $\M$ is computably categorical relative to $\mathbf{d}$?
\end{question}

The result of Goncharov \cite{Gon80EffDim} suggests that these questions are interesting when $\mathbf{d}$ is $\Delta_3$.
We suspect that our techniques can be used to negatively answer these questions in the case of $\Delta_2$ degrees with a bit more care applied to similar arguments.

Another direction to take this is to consider an alternating chain of behavior of given c.e. sets.
The following would mutually generalize the main result of \cite{DHTM21} and the results of this paper.

\begin{question}
    Given computably enumerable sets $X_0<_TY_0<X_1<_TY_1<\cdots$ is there a computable structure $\A$ that is computably categorical relative to the $X_i$ but not relative to the $Y_i$?
\end{question}

\printbibliography

\end{document}